\pdfoutput=1
\RequirePackage{ifpdf}
\ifpdf 
\documentclass[pdftex]{sigma}
\else
\documentclass{sigma}
\fi

\newcommand{\R}{\mathbb R}
\newcommand{\C}{\mathbb C}
\newcommand{\cC}{\mathcal C}
\newcommand{\cM}{\mathcal M}
\newcommand{\cS}{\mathcal S}
\newcommand{\cU}{\mathcal U}
\newcommand{\af}{\alpha}
\newcommand{\daf}{\dot{\alpha}}
\newcommand{\ga}{\gamma}
\newcommand{\dga}{\dot{\gamma}}
\newcommand{\ep}{\varepsilon}
\newcommand{\om}{\omega}
\newcommand{\Om}{\Omega}
\newcommand{\te}{\theta}
\newcommand{\rip}{\rangle}
\newcommand{\lip}{\langle}

\begin{document}
\allowdisplaybreaks

\newcommand{\arXivNumber}{1611.01573}

\renewcommand{\PaperNumber}{068}

\FirstPageHeading

\ShortArticleName{Null Angular Momentum and Weak KAM Solutions}

\ArticleName{Null Angular Momentum and Weak KAM Solutions\\ of the Newtonian $\boldsymbol{N}$-Body Problem}

\Author{Boris A. PERCINO-FIGUEROA}

\AuthorNameForHeading{B.A.~Percino-Figueroa}

\Address{Facultad de Ciencias en F\'isica y Matem\'aticas, Universidad Aut\'onoma de Chiapas, M\'exico}
\Email{\href{mailto:borispercino@yahoo.com.mx}{borispercino@yahoo.com.mx}}

\ArticleDates{Received November 09, 2016, in f\/inal form August 16, 2017; Published online August 24, 2017}

\Abstract{In [\textit{Arch. Ration. Mech. Anal.} \textbf{213} (2014), 981--991] it has been proved that in the Newtonian $N$-body problem, given a~minimal central conf\/iguration~$a$ and an arbitrary conf\/iguration $x$, there exists a completely parabolic orbit starting on $x$ and asymptotic to the homothetic parabolic motion of~$a$, furthermore such an orbit is a free time minimizer of the action functional. In this article we extend this result in abundance of completely parabolic motions by proving that under the same hypothesis it is possible to get that the completely parabolic motion starting at~$x$ has zero angular momentum. We achieve this by characterizing the rotation invariant weak KAM solutions as those def\/ining a lamination on the conf\/iguration space by free time minimizers with zero angular momentum.}

\Classification{37J15; 37J50; 70F10; 70H20}

\section{Introduction}

\subsection{Preliminaries}\label{prelim}

Let $E=\R^d$ be the $d$-dimensional Euclidean space, $d\geq2$, and consider the $N$-body problem with Newtonian potential function $U\colon E^N\to{} ]0,+\infty]$,
\begin{gather*}U(x)=\sum_{i<j}\frac{m_im_j}{r_{ij}},\end{gather*}
where $x=(r_1,\ldots,r_N)\in E^N$ is a conf\/iguration of~$N$ points having positive masses $m_1,\ldots,m_N$ in $E$, and $r_{ij}=|r_i-r_j|$. Here $|\cdot|$ denotes the Euclidean norm. We will denote by $\|\cdot\|$ to the norm induced by the mass inner product given by
\begin{gather*}x\cdot y=\sum_{i=1}^Nm_i\langle r_i, s_i\rangle,\end{gather*}
with $x=(r_1,\ldots,r_N)$, $y=(s_1,\ldots,s_N)\in E^N$, and $\langle \cdot, \cdot\rangle$ is the Euclidean inner product. We also introduce the moment of inertia:
\begin{gather*}I(x)= \|x\|^2.\end{gather*}

In this section we introduce the variational setting of the problem. The Lagrangian function $L\colon E^{2N}\to {} ]0,\infty]$ is given by
\begin{gather*}L(x,v)=\frac12I(v)+U(x)=\frac 12\sum_{i=1}^Nm_i|v_i|^2+U(x)\end{gather*}
and the action of an absolutely continuous curve $\ga\colon [a,b]\to E^N$ by
\begin{gather*}A_L(\ga)=\int_a^bL(\ga(t),\dga(t)){\rm d}t,\end{gather*}
so that the solutions of the problem are the critical points of the action functional.

For two given conf\/igurations $x,y\in E^N$, we will consider minima taken over the set $\cC(x,y)$ of absolutely continuous curves binding~$x$ and~$y$ without any restriction on time,
\begin{gather*}\cC(x,y):= \bigcup_{\tau>0}\big\{\ga\colon [a,b]\to E^N \text{ absolutely continuous }
b-a=\tau,\,\ga(a)=x,\, \ga(b)=y\big\}.\end{gather*}
The {\em Ma\~n\'e critical action potential} $\phi\colon E^n\times E^n\to[0,+\infty)$ is def\/ined as
\begin{gather*}\phi(x,y):=\inf\{A(\ga) \,\vert\, \ga\in\cC(x,y)\}.\end{gather*}

Here the inf\/imum is achieved if and only if $x\neq y$, this is essentially due to the lower semi-continuity of the action. Marchal's theorem asserts that minimizers avoid collisions in the interior of their interval of def\/inition \cite{Ch, Fe, Mar}.

Let us def\/ine the set of conf\/igurations without collision
\begin{gather}\label{nocol}
\Om:=\big\{x=(r_1,\dots,r_N)\in E^N\,|\, \text{if } r_i=r_j, \text{ then } i=j\big\},
\end{gather}
and let $M:=m_1+\dots+m_N$ the total mass of the system.

Given a conf\/iguration $x=(r_1,\dots,r_N)\in E^N$, the center of mass of $x$ is def\/ined by
 \begin{gather*}
G(x)=\frac{1}{M}\sum_{i=1}^{N}m_ir_i,
\end{gather*}
it is a standard fact that if $\ga\colon J\to E^N$ is a minimizer of the action, then $G(\ga(t))$ has constant velocity.

\begin{definition}
A {\em free time minimizer} def\/ined on an interval $J\subset\R$ is an absolutely continuous curve $\ga\colon J\to E^N$ which satisf\/ies $A(\ga|_{[a,b]})=\phi(\ga(a),\ga(b))$ for all compact subinterval $[a,b]\subset J$.
\end{definition}
An important example is given by A.~Da Luz and E.~Maderna in \cite{DM} where they proved that if~$a$ is a~\emph{minimal configuration}, i.e., a~minimum of the potential restricted to the sphere $I(x)=1$, then the parabolic homothetic ejection with central conf\/iguration $a$ is a free time minimizer. It is not known if there are other central conf\/igurations with this property. A.~Da Luz and E.~Maderna also proved that free time minimizers cannot be def\/ined in the whole line. On the other hand it is also proved that this minimizers are completely parabolic motions.

\subsection{Weak KAM solutions}

Weak KAM solutions of the Hamilton--Jacobi equation related to the problem are useful tools to study free time minimizers. The Hamiltonian associated to the problem is given by
\begin{gather*}
H(x,p)=\tfrac12||p||_*^2-U(x),
\end{gather*}
 where $||p||_*=\sum\limits_{i=1}^Nm_i^{-1}|p|$.

\begin{definition}
A {\em weak KAM solution} of the Hamilton--Jacobi equation
\begin{gather} \label{eq:hj}
 \|Du(x)\|_*^2=2U(x)
\end{gather}
is a function $u\colon E^N\to\R$ that satisf\/ies the following conditions:
\begin{itemize}\itemsep=0pt
\item $u$ is {\em dominated}, i.e., $u(y)-u(x)\le\phi(x,y)$ for all $x,y\in E^N$,
\item for any $x\in E^N$ there is an absolutely continuous curve $\af\colon [0,\infty)\to E^N$ such that $\af(0)=x$ and $\af$ {\em calibrates} $u$, i.e., $u(x)-u(\af(t))=A(\af|_{[0,t]})$ for any $t>0$.
\end{itemize}
\end{definition}

Notice that calibrating curves of weak KAM solutions are indeed free time minimizers. On the other hand, existence of weak KAM solution is proved in \cite{M}, where solutions are characterized as f\/ixed points of the so called Lax--Oleinik semigroup. A variety of weak KAM solutions is also obtained by means of Busemann functions used in Riemannian geometry and introduced in weak KAM theory by G.~Contreras~\cite{G} in the case of regular Hamiltonians; in~\cite{PS} it is proved the following proposition
\begin{proposition}\label{asymps} Let $a$ a minimal central configuration with $\|a\|=1$, define $U(a)=U_0$ and $c:=\big(\frac 92 U_0\big)^{\frac 13}$, consider the parabolic homothetic ejection with central configuration~$a$ given by $\ga_0(t)=ct^{\frac 23}a$. Then the Busemann function
 \begin{gather}\label{Bus}
 u_a(x) =\lim_{t\to+\infty} [\phi(x,\ga_0(t))-\phi(0,\ga_0(t)) ]
 \end{gather}
is a weak KAM solution of the Hamilton--Jacobi equation~\eqref{eq:hj}. Moreover, for any $x\in E^N$ there is a curve $\af\colon [0,\infty)\to E^N$ with $\af(0)=x$ that calibrates $u$ and
\begin{gather*}
\lim_{t\to+\infty}\big\|\af(t) t^{-\frac 23}-cx_0\big\|=0.
\end{gather*}
\end{proposition}

A solution def\/ined by identity \eqref{Bus}, will be called \emph{Busemann solution}. It is an open problem to determine if there
 are central conf\/igurations, dif\/ferent from minimal conf\/igurations, def\/ining Busemann solutions.

On the other hand, due to the symmetries of the potential function, it is interesting to determine if weak KAM solutions are invariant under this symmetries. In the case of translation invariance, E.~Maderna proved in \cite{M1} that given a weak KAM solution $u$ of \eqref{eq:hj}, then
\begin{gather*}
u(r_1,\dots,r_N)=u(r_1+r,\dots,r_N+r)
\end{gather*}
for any conf\/iguration $x=(r_1,\dots,r_N)\in E^N$ and every $r\in E$. The proof is achieved by showing that calibrating curves of weak KAM solutions have constant center of mass.

An important question is to determine if weak KAM solutions are rotation invariant, the main goal of this article is to study this problem. Notice that, there are solutions which are not rotation invariant, Busemann solutions given in~\eqref{Bus} for instance. Therefore the problem is to give conditions so that a weak KAM solution is rotation invariant; we achieve this goal by studying the angular momentum for the calibrating curves of rotation invariant solutions and characterizing invariant solutions as those where calibrating curves have zero angular momentum. We obtain rotation invariant solutions by setting
\begin{gather*}
 \hat u_a=\inf_{R\in {\rm SO}(d)}u_{R a}(x),
\end{gather*}
where $a$ is a minimal central conf\/iguration and $u_{R a}(x)$ is the Busemann function associated to~$Ra$.
\subsection{Main theorems}\label{teos}
We consider the diagonal group action on $E^N$ def\/ined by the special orthogonal group ${\rm SO}(d)$, more precisely, the rotation on $E^N$ by an element $\te\in {\rm SO}(d)$ is{\samepage
\begin{gather*}
R_{\te}\colon \ E^N \to E^N, \qquad x=(r_1,\dots,r_N) \mapsto(\te r_1,\dots,\te r_N),
\end{gather*}
where $\te r_i $ is the usual group action of ${\rm SO}(d)$ on~$E$.}

The \emph{Angular momentum} is a f\/irst integral closely related to the action of ${\rm SO}(d)$ on $E^N$. If $x=(r_1,\dots,r_N)\in E^N$ and a~vector $v=(v_1,\dots,v_N)\in E^N$ the angular momentum $C(x,v)$ is def\/ined as
\begin{gather*}
C(x,v)=\sum_{j=1}^Nm_j r_j\wedge v_j.
\end{gather*}
If $d=3$ the $\wedge$ product becomes the usual cross product in $E$. If $d=2$, by identifying $\R^2$ with~$\C$, if $x,v\in\C$ then $r\wedge v=\operatorname{Im}(v\bar r)$, and $r\wedge v$ is a real number.

Let $u\colon E^N\to\R$ a continuous function, we say that $u$ is \emph{rotation invariant} if for any $x\in E^N$ and any $\te\in {\rm SO}(d)$ we have
\begin{gather*}
u(R_\te(x))=u(x).
\end{gather*}

We have the following characterization of invariant weak KAM solutions of~\eqref{eq:hj} in terms of the angular momentum of their calibrating curves.

\begin{theorem}\label{Main1} Let $u$ be a weak KAM solution of the Hamilton--Jacobi equation
\begin{gather*}
 \|Du(x)\|_*^2=2U(x).
\end{gather*}
Then $u$ is rotation invariant if and only if all of its calibrating curves have zero angular momentum. That is to say, for any $\te\in {\rm SO}(d)$ and any $x\in E^N$, the identity
\begin{gather*}u(x)=u(R_{\te}x)\end{gather*}
holds if and only if for any $\ga\colon [0,+\infty[{}\to E^N$ calibrating $u$, we have
\begin{gather*}C(\ga(t),\dga(t))=0.\end{gather*}
\end{theorem}

We can give a more general result by considering $G$ a Lie group acting properly on $E$. Thus we can consider the diagonal action $S\colon G\times E^N\to E^N$ of $G$ on $E^N$, def\/ined by
\begin{gather*}S_gx=(g r_1,g r_2, \dots, g r_N),\end{gather*}
where $g\in G$, $x=(r_1,\dots,r_N)\in E^N$ and $g r$ is the action of $g$ on $r$. Let us denote by $\mathfrak g$ to the Lie algebra of $G$. We denote by $[\,,\,]$ to the pairing between $\mathfrak g$ and $\mathfrak g^*$.

Notice that the action $S$ can be lifted to $E^N\times E^N$ by $g(x,v)=(S_gx, T_xS_g v)$ where $T_xS_g$ is the dif\/ferential of $S_g$ at $x$. Assume that the Lagrangian is $G$-invariant, i.e., $g^*L=L$, $g\in G$ and assume also that the action lifts to $E^N\times E^N$ by isometries of the mass inner product.

Under such conditions, the group action def\/ines an \emph{equivariant momentum map}
\begin{gather*}
\mu\colon \ E^N\times E^N\to \mathfrak g^*,
\end{gather*}
given by
\begin{gather}\label{equi:mom}
[\mu(x,v),\xi]=v\cdot X_{\xi}(x),
\end{gather}
where $X_{\xi}(x)=\frac{{\rm d}}{{\rm d}t}\bigr|_{t=0}S_{\exp(t\xi)}x$ is the inf\/initesimal generator of the one-parameter subgroup action on~$E^N$, associated to $\xi\in\mathfrak g$.

A continuous function $u\colon E^N\to\R$ is \emph{$G$-invariant} if for any $g\in G$ and any $x\in E^N$ we have
\begin{gather*}
 u(S_gx)=u(x).
\end{gather*}
In a similar way to Theorem \ref{Main1} we can give a characterization to $G$-invariant weak KAM solutions in terms of the equivariant momentum map.

\begin{theorem}\label{Main2}
Let $G$ a connected Lie group acting diagonally on $E^N$ and suppose that the group action satisfies the assumptions above, and let $u$ be a weak KAM solution of the Hamilton--Jacobi equation \eqref{eq:hj}. Then $u$ is $G$-invariant if and only if for any $\ga\colon [0,+\infty[{}\to E^N$ that calibrates $u$, for any $t>0$, we have
\begin{gather*}\mu(\ga(t),\dga(t))=0.\end{gather*}
\end{theorem}

\section{Rotation invariance}

Given $x\in E^N$, consider the orbit of $x$ under ${\rm SO}(d)$ given by
\begin{gather*}M_x:=\{R_\te x\,|\,\te\in {\rm SO}(d)\},\end{gather*}
 let us remind that
\begin{gather*}T_xM_x=\{Ax\,|\,A\in \mathfrak{so}(d)\}.\end{gather*}
The key point in the proof of Theorem~\ref{Main1} is the Saari decomposition of the velocities \cite{Ch, S}. Def\/ine
\begin{gather*}\mathcal{H}_x:=\big\{v\in E^N\,|\,C(x,v)=0\big\}.\end{gather*}
Then $T_xM_x\perp\mathcal{H}_x$, with respect to the mass scalar product and
\begin{gather*}
E^N=T_xM_x\oplus\mathcal{H}_x.
\end{gather*}
In other words, if $v\in E^N$, then $v$ can be decomposed as
\begin{gather*}v=v_r+v_h,\end{gather*}
where $v_r\in T_xM_x$, $C(x,v_h)=0$ and $v_r\cdot v_h=0$, moreover the components $v_r$ and~$v_h$ are uniquely determined by~$v$. For dimensions~2 and~3 this is a direct consequence of the properties of the cross product, for dimensions $\geq4$ it is due to the properties of the ``wedge'' product.

On the other hand notice that if $u$ is a rotation invariant function and $x\in E^N$, then $u$ is constant on $M_x$. Thus, if $x$ is a point of dif\/ferentiability of~$u$, we have that
\begin{gather*} T_xM_x\subset \ker {\rm d}_xu.\end{gather*}

\begin{proof}[Proof of Theorem \ref{Main1}] Let $u$ be a rotation invariant weak KAM solution and let $x\in E^N$, consider a curve $\ga\colon [0,+\infty[{} \to E^N$ calibrating $u$ and starting at $x$. It is known that $u$ is dif\/ferentiable at $\ga(t)$ for any $t>0$, we also know that
\begin{gather}\label{leg}
{\rm d}_{\ga(t)}u(w)= w\cdot\dga(t),
\end{gather}
for all $w\in E^N$.

On the other hand, by the previous remark we have that $T_{\ga(t)}M_{\ga(t)}\subset \ker {\rm d}_{\ga(t)}u$ and from~\eqref{leg} we get $\dga(t)\perp T_{\ga(t)}M_{\ga(t)}$, therefore $\dga(t)\in\mathcal{H}_{\ga(t)}$, then{\samepage
\begin{gather*}
C(\ga(t),\dga(t))=0
\end{gather*}
for all $t>0$.}

Let us consider now a weak KAM solution $u$ such that all of its calibrating curves have zero angular momentum. Let $x\in E^N$ and let $\te\in {\rm SO}(d)$. We will prove that $u(R_{\te}x)=u(x)$.

Clearly if $R_{\te}x=x$, the result follows trivially. Suppose that $R_{\te}x\neq x$, since $u$ is continuous and the set of collisionless conf\/igurations~$\Om$, given in~\eqref{nocol}, is open, dense and rotation invariant, we can assume $x\in\Om$.

Since ${\rm SO}(d)$ is compact, $\exp\colon \mathfrak{so}(d)\to {\rm SO}(d)$ is surjective, thus we can take $\om\in\mathfrak{so}(d)$ such that $\exp(\om)=\te$. Def\/ine the curve $\af\colon [0,1]\to E^N$ by $\af(t)=R_{\exp(t\om)}(x)$. Let $\ep>0$ be small enough so that
\begin{gather*}
B:=\big\{z\in E^N\,|\,\lip z,\daf(0) \rip=0,\, \|z-x\|<\ep\big\}\subset\Om.
\end{gather*}
Notice that the set $W:=\{z\in E^N\,|\,R_{\te}z\neq z\}$ is open, therefore we can also choose $\ep>0$, smaller if necessary, so that $B\subset W$.

We can assume that $\ep$ is suf\/f\/iciently small so that the map
\begin{gather}\label{difeo}
B\times [0,1]\to E^N, \qquad (z,t)\mapsto R_{\exp(tw)}z
\end{gather}
is a dif\/feomorphism onto its image. Indeed, If $\af(t)=R_{\exp(t\om)}(x)\neq x$ for every $t\in(0,1]$, then the curve $\af$ is an embedding and, by choosing $\ep$ suf\/f\/iciently small, the map \eqref{difeo} is a~dif\/feomorphism onto its image. Suppose, on the contrary, that $\af(\tau)=x$ for some $\tau\in(0,1]$, then $\af$ is $\tau$-periodic. Let $\tau$ be the minimal period of $\af$, then there exists $s\in(0,\tau)$ such that $\af(s)=R_{\exp(s\om)}(x)=R_{\exp(\om)}(x)$ and~$\af$ has no self intersections in the interval~$[0,s]$. Replacing~$\om$ by~$s\om$, and choosing~$\ep$ suf\/f\/iciently small we get as before that \eqref{difeo} is a dif\/feomorphism onto its image.

Def\/ine
\begin{gather*}C:=\big\{R_{\exp(t\om)}(z)\,|\,z\in B,\,t\in[0,1]\big\}\end{gather*}
and denote by $C'$ to the set point in $C$ where $u$ is dif\/ferentiable. Since $u$ is dominated, $u$ is Lipschitz in $C$ and therefore~$C'$ has total measure in $C$. Notice also that $C\subset\Om$.

Let $y\in C'$ and let $\ga_y\colon [0,+\infty[{}\to E^N$ be a calibrating curve such that $\ga(0)=y$, then ${\rm d}_yu(w)=w\cdot\dga_y(0)$ for any $w\in E^N$. From the hypothesis $\dga_y(0)\in\mathcal{H}_q$, thus $\dga_y(0)\perp T_yM_y$, thus $T_yM_y\subset\ker {\rm d}_yu$ for any $y\in C'$.

Let $f\colon B\times[0,1]\to\R$ the Lipschitz continuous function given by
\begin{gather*}
f(z,t)=u\big(R_{\exp(t\om)}(z)\big)-u(z).
\end{gather*}
Since $T_yM_y\subset\ker {\rm d}_yu$, we get $\frac{\partial f}{\partial t}=0$ almost everywhere, by Fubini theorem in $A\times[0,1]$, with~$A$ any open subset of $B$
\begin{gather*}
0=\int_A\int_{[0,1]}\frac{\partial f}{\partial t}{\rm d}t{\rm d}y=\int_Af(y,1){\rm d}y,
\end{gather*}
therefore $f(y,1)=0$ for every $y\in B$, in particular we have
\begin{gather*}u(R_\te(x))=u(x).\tag*{\qed}\end{gather*}\renewcommand{\qed}{}
\end{proof}

\begin{remark}\label{inf;sol} Let us denote by $\cS$ to the set of weak KAM solutions of \eqref{eq:hj} and notice that if $\cU\subset \cS$ is such that $\underset{u\in\cU}\inf u(x)>-\infty$, then
\begin{gather*}
\tilde u(x)=\underset{u\in\cU}\inf \{u(x)\,|\, u\in \cU\}
\end{gather*}
is in $\cS$. This is due to the fact that weak KAM solutions are the f\/ixed points of the Lax--Oleinik semigroup \cite{M}.
\end{remark}

\begin{corollary}\label{invKAM} Let $a$ be a~minimal central configuration with $I(a)=1$. For any $\te\in {\rm SO}(d)$, let~$u_{R_\te a}$ be the Busemann solution associated to the minimal central configuration~$R_\te a$. Then the function
\begin{gather*}
\hat u_a \colon \ E^N\to\R, \qquad \hat u_a(x):=\inf_{\te\in {\rm SO}(d)}u_{\te,a}(x)
\end{gather*}
is a rotation invariant weak KAM solution of the Hamilton--Jacobi equation. Therefore the callibrating curves of $\hat u_a$ are free time minimizers having zero angular momentum.
\end{corollary}

\begin{proof}Let $\cM$ be the set of minimal central conf\/igurations with moment of inertia one. Let $a\in\cM$ and $\te\in {\rm SO}(d)$ let
\begin{gather*}
u_{\te,a}(x)=u_a\circ R_{\te^{-1}}(x),
\end{gather*}
it is not hard to see that $u_{\te,a}$ is also a weak KAM solution, furthermore, notice that $u_{R_\te a}=u_{\te,a}$, thus
\begin{gather}\label{uhat}
 \hat u_a(x)=\inf_{\te\in {\rm SO}(d)}u_{\te,a}(x),
 \end{gather}
and from the previous remark the function on the right is a weak KAM solution.

Due to \eqref{uhat}, $\hat u_a$ is rotation invariant and from Theorem~\ref{Main1}, these solutions def\/ine laminations by free time minimizer with zero angular momentum.
\end{proof}

Given a minimal central conf\/iguration $a$, notice that the rotation invariant weak KAM solution $\hat u_a$ given in the previous Corollary in uniquely determined by~$M_a$, the orbit of~$a$ under~${\rm SO}(d)$, we call this solution \emph{invariant Busemann solution} associated to~$M_a$.

\section[$G$-invariance]{$\boldsymbol{G}$-invariance}

Let $G$ be a connected Lie group acting on $E^N$ with the assumptions of Section~\ref{teos}, let us notice that in this setting, due to~\eqref{equi:mom} the equivariant momentum map def\/ines a Saari decomposition of the velocity (see \cite{AMR,A,MHO}), as follows.

For a f\/ixed momentum value, $\mu(x,v)=\mu$, there are orthogonal vectors $v_{\mathcal H}$ and $v_{\mathcal V}$ such that
\begin{gather*}
v=v_{\mathcal H}+v_{\mathcal V},\qquad
\mu(x,v_{\mathcal V})=\mu \qquad \text{and} \qquad \mu(x,v_{\mathcal H})=0.
\end{gather*}
Let $x\in E^N$ and let $G_x$ be the orbit of $x$ under the $G$-action. Consider the subspaces of $T_xE^N=E^N$
\begin{gather*}
\mathcal{H}{\rm or}_x=\big\{v\in E^N\,|\,\mu(x,v)=0\big\} \qquad \text{and} \qquad T_xG_x,
\end{gather*}
then these subspaces are orthogonal with respect to the mass inner product and
\begin{gather*}
E^N=\mathcal{H}{\rm or}_x\oplus T_xG_x.
\end{gather*}
Thus, any $v\in E^N$ can be uniquely decomposed as
\begin{gather}\label{saarigen}
v=v_{\mathcal H}+v_{\mathcal V},
\end{gather}
where $v_{\mathcal H}\in\mathcal{H}{\rm or}_x$, $v_{\mathcal V}=X_\xi\in T_xG_x$, and $\xi\in \mathfrak g$ is the a element such that $\mu(x,v)=\mu(x,X_\xi)$.

Finally notice that if $u$ is a $G$-invariant function and $x$ a point of dif\/ferentiability of $u$, then $T_xG_x\subset {\rm d}_xu$.

\begin{proof}[Proof of Theorem \ref{Main2}] The main dif\/f\/iculty is the surjectivity of the exponential map, nevertheless it can be avoided as follows. Since $G$ is connected, it is well known that for any $g\in G$, there exists $\xi_1,\dots,\xi_n \in \mathfrak g$ such that $g=\exp(\xi_1)\cdots\exp(\xi_n)$. Therefore, if we can prove that
\begin{gather}\label{restr}
u\big(S_{\exp(\xi)}x\big)=u(x)
\end{gather}
for any $x\in E^n$ and any $\xi\in\mathfrak g$, we get that $u(S_g x)=u(x)$ for any $x\in E^n$ and any $g\in G$. Given the Saari decomposition of the velocities~\eqref{saarigen}, the proof of~\eqref{restr} follows, as the one of Theo\-rem~\ref{Main1}
\end{proof}

We can apply Theorem \ref{Main2} to any connected subgroup $G\subset {\rm SO}(d)$ getting that a solution is $G$-invariant if and only if the corresponding component of the angular momentum of the calibrating curves, in the direction of $\mathfrak g^*$ is null at any instant of the motion.

\subsection*{Acknowledgements}
The author acknowledges the referees for their valuable suggestions and remarks that substantially improved this article. The author is grateful to Ezequiel Maderna for his advice, to H\'ector S\'anchez Morgado for his suggestions and to Eddaly Guerra Velasco for her support in the process of this research.

\pdfbookmark[1]{References}{ref}
\LastPageEnding

\end{document}